\documentclass[letterpaper,11pt,reqno]{amsart}
\usepackage{amssymb, amsmath, amsfonts, amscd}
\usepackage{mathrsfs, latexsym}
\usepackage[all,ps,cmtip]{xy}
\usepackage{array, longtable}
\usepackage{bm, enumitem}
\usepackage{xcolor}
\usepackage{bm}
\usepackage{soul}
\usepackage{mathtools}

\usepackage{tikz-cd}
\usepackage{tikz}
\usepackage{graphicx}
\usepackage{calc}
\usepackage{accents} 
\allowdisplaybreaks
\usepackage{amsthm}
\usepackage{hyperref}

\numberwithin{equation}{section}

\newtheorem{theorem}{Theorem}[section]
\newtheorem{lemma}[theorem]{Lemma}

\newtheorem{proposition}[theorem]{Proposition}
\newtheorem{def-prop}[theorem]{Definition-Proposition}
\newtheorem{claim}[theorem]{Claim}

\theoremstyle{definition}
\newtheorem{definition}[theorem]{Definition}

\newtheorem{example}[theorem]{Example}

\newtheorem{remark}[theorem]{Remark}
\newtheorem{setting}[theorem]{Setting}
\newtheorem{op}[theorem]{Open Question}

\DeclareMathOperator{\Ivol}{Ivol}

\DeclareMathOperator{\Mov}{Mov}

\newcommand{\bQ}{{\mathbb Q}}

\title{Sharp Noether-type inequalities for $3$-folds of intermediate Kodaira dimensions}

\date{\today} 

\author{Meng Chen}
\address{\rm School of Mathematical Sciences \& Shanghai Center for Mathematical Sciences, Fudan University, Shanghai 200433, China}
\email{mchen@fudan.edu.cn}

\author{Yong Hu}
\address{\rm School of Mathematical Sciences, Shanghai Jiao Tong University, Shanghai 200240, China}
\email{yonghu@sjtu.edu.cn}

\author{Chen Jiang}
\address{\rm Shanghai Center for Mathematical Sciences \& School of Mathematical Sciences, Fudan University, Shanghai 200438, China}
\email{chenjiang@fudan.edu.cn}

\author{Yang Liu}
\address{\rm Shanghai Center for Mathematical Sciences, Fudan University, Shanghai 200438, China}
\email{23110840006@m.fudan.edu.cn}
\email{yangliu70223@gmail.com}

\begin{document}

\begin{abstract}
For a smooth projective $3$-fold $V$ of intermediate Kodaira dimension, we establish Noether-type inequalities relating the Iitaka volume $\operatorname{Ivol}(V)$ and the geometric genus $p_g(V)$. Specifically, for a smooth projective $3$-fold $V$ of Kodaira dimension $2$, we show that
$\operatorname{Ivol}(V)\geq \frac{3}{4}p_g(V)-\frac{5}{4};$
while if the Kodaira dimension of $V$ is $1$, then
$\operatorname{Ivol}(V)\geq p_g(V)-1.$
Both inequalities are sharp. In the Kodaira dimension $2$ case, we prove that the equality holds only if the bicanonical system of $V$ induces the Iitaka fibration and the canonical system of $V$ is composed with a pencil of surfaces whose general member has Iitaka volume $\frac12$ provided that $p_g(V)\geq 4$.
\end{abstract}


\maketitle

\pagestyle{myheadings}
\markboth{\hfill M. Chen, Y. Hu, C. Jiang, Y. Liu \hfill}{\hfill Sharp Noether-type inequalities for $3$-folds of intermediate Kodaira dimensions \hfill}

\section{Introduction}
Throughout this paper we work over the complex number field $\mathbb{C}$.

The classical Noether inequality lies at the heart of the birational classification of projective surfaces. Originally formulated by M. Noether (\cite{N1870}) in 1870, this fundamental result asserts that any minimal projective surface $S$ of general type satisfies the sharp inequality: $K_S^2 \geq 2p_g(S)-4$. As a powerful quantitative constraint linking two core birational invariants, this inequality serves as a crucial tool for the classification of surfaces of general type. To extend this result to broader geometric settings, Tsunoda and Zhang \cite{T-Z1}, together with Zhang’s further independent studies \cite{Z2}, established a series of Noether-type inequalities for non-complete algebraic surfaces, which generalize the classical Noether theory to algebraic surfaces of non-general type.

In higher-dimensional geometry, the generalization of Noether inequalities to threefolds has long been a central open problem. Very recently, Chen, Chen, and Jiang \cite{CCJ1, CCJ2}, as well as Chen, Hu, and Jiang \cite{CHJ25}, established the sharp inequality $\operatorname{vol}(V)\geq \frac{4}{3}p_g(V)-\frac{10}{3}$ for all smooth projective threefolds $V$ of general type. Nevertheless, the classification theory for threefolds of non-general type remains far from complete and largely unexplored. In particular, systematic Noether-type inequalities relating the fundamental birational invariants of threefolds are still missing, which hinders further progress in the geometric classification of non-general type threefolds. Focusing on smooth projective threefolds $V$ of intermediate Kodaira dimension, this paper aims to establish new sharp Noether-type inequalities connecting the Iitaka volume $\operatorname{Ivol}(V)$ (see \eqref{Ivol} for its definition) and the geometric genus $p_g(V)$.

Our main results are as follows. 

\begin{theorem}\label{thm:main1}
Let $V$ be a smooth projective $3$-fold with $\kappa(V)=2$. Then \[
\Ivol(V)\geq \frac{3}{4}p_g(V)-\frac{5}{4}.
\] 
\end{theorem}

\begin{theorem}\label{thm:main2}
Let $V$ be a smooth projective variety with $\kappa(V)=1$. 
Then \[
\Ivol(V)\geq p_g(V)-1.
\]
\end{theorem}

Both inequalities are optimal 
by Example~\ref{example: pencil} and Example~\ref{example: kodaira dim 1}.
These two results fill the gap in Noether-type quantitative bounds for threefolds with intermediate Kodaira dimensions $\kappa=1,2$. Furthermore, the newly derived inequalities relating the Iitaka volume and geometric genus offer novel effective quantitative criteria and a rigorous analytical framework for explicit birational classification of threefolds of intermediate Kodaira dimensions.

We can also describe the equality case of Theorem~\ref{thm:main1}. 

\begin{theorem}\label{thm:equality case}
Let $X$ be a minimal projective $3$-fold with $\kappa(X)=2$ such that $p_g(X)\geq 2$ and \[
\Ivol(X)= \frac{3}{4}p_g(X)-\frac{5}{4}.
\] 
Then one of the following statements holds:
\begin{enumerate}

 \item $p_g(X)=3$ and $|K_X|$ is base point free and induces a morphism \[\Phi_{|K_X|}: X\to \mathbb{P}^{2}\] which coincides with the Iitaka fibration of $X$;

 \item $p_g(X)\geq 3$ and the following properties hold:
\begin{enumerate}
\item $X$ is Gorenstein and $h^1(X, \mathcal{O}_{X})=h^2(X, \mathcal{O}_{X})=0$;
\item $|2K_X|$ is base point free with $h^0(X, 2K_X)=3p_g(X)-3$ and induces a morphism \[\Phi_{|2K_X|}: X\to T_2\subseteq \mathbb{P}^{3p_g(X)-4}\] where $T_2$ is a surface of minimal degree $\deg T_2=3p_g(X)-5$ and $\Phi_{|2K_X|}$ coincides with the Iitaka fibration of $X$;

\item $|K_X|$ is composed with a pencil of surfaces $G$ with \[\Ivol(G)=\begin{cases}\frac{1}{2}  \text{ or } \frac34 &\text{if } p_g(X)=3;\\
\frac{1}{2} &\text{if } p_g(X)\geq 4.
\end{cases}\]
\end{enumerate}
\end{enumerate}
\end{theorem}
We may compare Theorem~\ref{thm:equality case} with the corresponding results for minimal threefolds of general type. Let $Y$ be a minimal threefold of general type on the Noether line with $p_g(Y)\geq 5$. 
\begin{enumerate}
 \item Hu and Zhang \cite{HZ} proved that $Y$ is Gorenstein, $h^1(Y,\mathcal{O}_Y)=h^2(Y, \mathcal{O}_Y)=0$, and the canonical image of $Y$ is a surface of minimal degree;
 \item Coughlan, Hu, Pignatelli, and Zhang \cite{CHPZ} proved that $Y$ birationally admits a simple fibration in $(1,2)$-surfaces over $\mathbb{P}^1$.
\end{enumerate}
Theorem~\ref{thm:equality case} show that similar results hold for minimal $3$-folds with $\kappa=2$. 

The proof of Theorem~\ref{thm:main2} is quite standard and may be well-known to experts, while Theorem~\ref{thm:main1} is much harder. We briefly explain the difficulties and ingredients in the proof of Theorem~\ref{thm:main1}.
According to the minimal model program, we may replace $V$ by its minimal model $X$. Let us consider the most difficult case that the canonical map $\Phi_{|K_X|}$ is composed with a pencil.
One must overcome at least three essential technical obstacles. First, it is necessary to establish a comparison inequality between $\phi^*K_X|_G$ and $\sigma^*K_{G_0}$ of the form
$\phi^*K_X|_G\geq c\,\sigma^*K_{G_0}$
for some constant $c>0$, where $\phi:X'\rightarrow X$ denotes a birational modification, $g:X'\rightarrow \Gamma$ is a fibration over a smooth curve, $G$ stands for a general fiber of $g$, and $\sigma:G\rightarrow G_0$ is the contraction onto the minimal model of $G$. Although $K_X$ is not big, the terminal model construction for the canonical pencil of $X$ developed by Chen, Hu, and Jiang \cite[\S 3]{CHJ25} remains valid in this setting. The second difficulty arises from the lack of effective vanishing theorems when analyzing the linear series $|mK_X|$, especially concerning its restriction $|mK_X||_G$. To resolve this issue, we apply the injectivity theorem from \cite[Theorem~6-1-6]{KMM87}. 
Thirdly, even if the comparison inequality between $\phi^*K_X|_G$ and $\sigma^*K_{G_0}$ is established, it can not give the optimal Noether-type inequality when the surface $G$ has small Iitaka volume. To resolve this issue, first we  
give a precise description of elliptic surfaces with $\kappa=1$, $p_g>0$, and extremely small Iitaka volumes, then we show that if the canonical map is fibered by such surfaces, then the bi-canonical system $|2K_X|$ behaves better and sharpens the estimate of the Iitaka volume of $X$.
These ingredients together enable us to establish the sharp Noether-type inequality in dimension three.

\section{Preliminaries}

\subsection{Notation and conventions}
We adopt the standard notation and definitions from \cite{KMM87, Kollar-Mori-1998, Kawamata_2024} and use them freely.

A $\mathbb{Q}$-divisor $D$ on a normal variety $X$ is said to be {\it $\mathbb{Q}$-effective} if it is $\mathbb{Q}$-linearly equivalent to an effective $\mathbb{Q}$-divisor.

A {\it fibration} is a projective morphism $f:X\to Z$ between normal varieties such that $f_*\mathcal{O}_X=\mathcal{O}_Z$. In particular, $f$ is surjective with connected fibers.

\subsection{Iitaka volume and Iitaka fibration}
We recall basic definitions of
 Kodaira dimension, Iitaka volume, and Iitaka fibration. 
 
\begin{definition}
Let $X$ be a normal projective variety with at worst terminal singularities.
\begin{enumerate}
 \item The {\it geometric genus} of $X$ is defined as 
\[
p_g(X)\coloneqq h^0(X,K_X).
\]
\item The {\it Kodaira dimension} $\kappa(X)$ of $X$ is defined as follows: if $|mK_X|\neq \emptyset$ for some positive integer $m$, then 
\begin{equation*}
 \kappa(X)\coloneqq 
 \max \{\dim\Phi_{|mK_X|}(X)\mid m\in\mathbb{Z}_{\geq1}\},
\end{equation*}
where $\Phi_{|mK_X|}$ is the rational map defined by $|mK_X|$; otherwise, $\kappa(X)\coloneqq -\infty$. 
\end{enumerate}

\end{definition}

\begin{definition}Let $X$ be a normal projective variety with at worst terminal singularities and with $\kappa(X)=\kappa> 0$.
The {\it Iitaka volume} $\Ivol(X)$ of $X$ is defined as
\begin{equation}
\Ivol(X)=\limsup_{m\to +\infty} \frac{h^0(X,mK_X)}{m^{\kappa}/\kappa!}. \label{Ivol}
\end{equation}
\end{definition}

It is well-known that
for a normal projective variety $X$ with at worst terminal singularities, $p_g(X), \kappa(X)$, and $\Ivol(X)$ are all birational invariants. So we are free to replace $X$ by any of its birational models. 
According to the minimal model program, every smooth projective $3$-fold $V$ with $\kappa(V)\geq 0$ is birational to a {\it minimal} projective $3$-fold $V_{\min}$ such that $V_{\min}$ has $\mathbb{Q}$-factorial terminal singularities and $K_{V_{\min}}$ is semiample (see \cite{Kaw92}).

\begin{definition}
 Let $X$ be a normal projective variety with at worst terminal singularities such that $\kappa(X)> 0$ and $K_X$ is semiample. Then there is a fibration 
\[
f:X\to Z
\]
and 
an ample $\mathbb{Q}$-divisor $L$ on $Z$ such that $K_X\sim_{\mathbb{Q}}f^*L$. In this case,
we have $\dim Z=\kappa(X)$ and
\[
\Ivol(X)=L^{\dim Z}.
\]
The fibration $f:X\to Z$ is called the {\it Iitaka fibration} of $X$. 
\end{definition}

\subsection{Pluricanonical maps}
We will often use the following setting (especially for $m=1$) in the study of pluricanonical maps of $3$-folds.
\begin{setting}\label{setting}Let $X$ be a minimal projective $3$-fold with 
$\kappa(X)>0$ and let $f:X\to Z$ be the Iitaka fibration with $K_X\sim_{\mathbb{Q}} f^*L$ for some ample $\mathbb{Q}$-divisor $L$ on $Z$. 
For a positive integer $m$ with $h^0(X, mK_X)\geq 2$,
consider the $m$-canonical map 
\[
\Phi_{|mK_X|}:X\dashrightarrow T_m\subseteq \mathbb{P}^{h^0(X, mK_X)-1}
\]
defined by the linear system $|mK_X|$, where $T_m$ is the closure of its image.
As $mK_X$ is numerically trivial and effective on a general fiber of $f$, it is linearly trivial on a general fiber of $f$. Hence
$\Phi_{|mK_X|}$ factors through $f$. We will construct a commutative diagram as follows: 
\begin{equation*}
\begin{tikzcd}
X' \arrow[r, "f'"] \arrow[d, "\phi"'] \arrow[rr, "g_m", bend left, shift left] & Z' \arrow[d, "\phi'"'] \arrow[r, "g'_m"] 
& \Gamma_m \arrow[d, "s_m"] & \\
X \arrow[r, "f"] \arrow[rr, "\Phi_{|mK_X|}", dashed, bend right, shift right] & Z \arrow[r, dashed] & T_m \arrow[r, hook] & \mathbb{P}^{h^0(X, mK_X)-1}
\end{tikzcd}
\end{equation*}

Let $\phi':Z'\to Z$ be a resolution of $Z$ resolving the natural rational map
$Z\dashrightarrow T_m$ and let $Z'\xrightarrow{g'_m}\Gamma_m\xrightarrow{s_m} T_m$ be the Stein factorization. Let $\phi:X'\to X$ be a resolution of $X$ such that 
there is a morphism $f':X'\to Z'$ with $f\circ \phi= \phi'\circ f'$. Denote by $g_m: X'\to \Gamma_m$ the natural morphism.
Take $L'_m$ to be a general member in 
$g'^*_ms_m^*|\mathcal{O}_{T_m}(1)|$. 
Then we have
\begin{align}\label{eq:f*L=Mov}
|f'^*L'_m|=\Mov|\lfloor\phi^*(mK_X)\rfloor|
\end{align}
since the natural morphism $X'\to T_m$ is defined by $\Mov|\lfloor\phi^*(mK_X)\rfloor|$. 

We will always omit the subscript $m$ when $m=1$.
\end{setting}

\begin{definition}\label{def:pencil}
 Keep the notation in Setting~\ref{setting}. We say that $|mK_X|$ is {\it composed with a pencil of surfaces} if $\dim T_m=1$ (or equivalently, $\dim \Gamma_m=1$). Sometimes we will say that $|mK_X|$ is {\it composed with a pencil of surfaces $G$}, where $G$ is a general fiber of $g_m: X'\to \Gamma_m$. 
\end{definition}

\begin{remark}\label{rem:pencil}
 In the setting of Definition~\ref{def:pencil}, we have the following useful facts:
 \begin{enumerate}
 \item $G$ is a smooth projective surface with $\kappa(G)=\kappa(X)-1$ by \cite[Proposition~1]{Fuj77}. 
 
 \item If $M_m\in \Mov|\lfloor\phi^*(mK_X)\rfloor|$, then we
 may write $M_m\equiv aG$ for some integer \[a\geq h^0(X, mK_X)-1,\] and the equality holds if and only if $p_g(\Gamma_m)=0$. 
 
 Indeed, $\Mov|\lfloor\phi^*(mK_X)\rfloor|=|g_m^*L_{\Gamma_m}|$ where $L_{\Gamma_m}$ is a base-point free divisor on $\Gamma_m$ with $h^0(\Gamma_m, L_{\Gamma_m})=h^0(X, mK_X)$, which implies that \[a=\deg L_{\Gamma_m}\geq h^0(\Gamma_m, L_{\Gamma_m})-1=h^0(X, mK_X)-1,\]
 where the inequality follows from
 Lemma~\ref{lem:deg>=h0-1} below and the equality holds if and only if $p_g(\Gamma_m)=0$. 
 \end{enumerate} 
\end{remark}
\begin{lemma}\label{lem:deg>=h0-1}
 Let $\mathcal{L}$ be a line bundle on a smooth projective curve $C$ with $h^0(C, \mathcal{L})\geq 2$. Then $\deg\mathcal{L}\geq h^0(C, \mathcal{L})-1,$ and the equality holds if and only if $p_g(C)=0$. 
\end{lemma}
\begin{proof}
 By the Riemann--Roch formula, we have
 \[
 \deg \mathcal{L}=h^0(C, \mathcal{L})-h^0(C, \mathcal{L}^{-1}\otimes \omega_C)+h^0(C, \omega_C)-1.
 \]
 Then the inequality follows from the fact that $h^0(C, \omega_C)\geq h^0(C, \mathcal{L}^{-1}\otimes \omega_C)$ as $h^0(C, \mathcal{L})>0$. Recall that $|\omega_C|$ is base-point free if $p_g(C)>0$, so the equality holds if and only if $h^0(C, \omega_C)= h^0(C, \mathcal{L}^{-1}\otimes \omega_C)$, if and only if $p_g(C)=0$.
\end{proof}

We will often use the following lemma to compare $L$ and $L'_m$.
\begin{lemma}\label{lem:compare L and L',m}
Keep the notation in Setting~\ref{setting}. 
Then $m\phi'^*L- L'_m$ is $\mathbb{Q}$-effective on $Z'$. In particular, $(mL)^{\dim Z}\geq (L'_m)^{\dim Z}$.
\end{lemma}
\begin{proof}

 Note that 
\[f'^*(m\phi'^*L-L'_m)\sim_\mathbb{Q} m\phi^*K_X-f'^*L'_m\]
is $\mathbb{Q}$-effective by \eqref{eq:f*L=Mov}. Hence $m\phi'^*L- L'_m$ is $\mathbb{Q}$-effective as $f'$ is a fibration. 
As $\phi'^*L$ and $L'_m$ are nef divisors on $Z'$, we have $(mL)^{\dim Z}\geq (L'_m)^{\dim Z}$.
\end{proof}

As a quick application, we are prepared to prove Theorem~\ref{thm:main2}.

\begin{proof}[Proof of Theorem~\ref{thm:main2}]
We may assume that $p_g(V)\geq 2$, otherwise the inequality is trivial.
Consider the canonical map $\Phi_{|K_V|}: V\dashrightarrow T\subseteq \mathbb{P}^{p_g(V)-1}$
where $T$ is the closure of its image. 
Take a resolution $V'$ of $V$ 
with morphisms $\phi: V'\to V$ and $h: V'\to T$.
Then 
$\Mov|\lfloor\phi^*K_V\rfloor|=|h^*\mathcal{O}_T(1)|.$
So we have $h^0(V , mK_V)\geq h^0(T, \mathcal{O}_T(m))$ for every positive integer $m$.
Since $\kappa(V)=1$ and $\dim T>0$, we know that $\dim T=1$. 
Moreover, by the Riemann--Roch formula on curves, we have
\begin{align*}
\Ivol(V)\geq \limsup_{m\to +\infty}\frac{h^0(T, \mathcal{O}_T(m))}{m}=\deg_T \mathcal{O}_T(1)\geq p_g(V)-1.
\end{align*}
Here the last inequality follows from the non-degeneracy of $T$ (\cite[Corollary~18.12]{Harris1992AlgebraicGeometryFirstCourse}).
This concludes the proof.
\end{proof}

\subsection{Restriction comparison result for canonical divisors}
In this subsection, we establish a restriction comparison result for canonical divisors. 
It was essentially proved in \cite[Lemma~3.3]{CHJ25} during the study of the Noether inequality of $3$-folds of general type. 
The proof below closely follows the original argument with necessary modifications.

\begin{lemma}[{cf. \cite[Proposition~3.1, Lemma~3.3]{CHJ25}}]\label{lem:compare KX KG}
Let $X$ be a minimal projective $3$-fold.
Let $\phi:X'\to X$ be a projective birational morphism where $X'$ is terminal, such that there is a fibration $g:X'\to \Gamma$ where $\Gamma$ is a smooth projective curve. 
Let $G$ be a general fiber of $g$ and $\sigma:G\to G_0$ the morphism onto its minimal model. 

\begin{enumerate}
 \item If $p_g(\Gamma)>0$, then $\phi^*K_X|_G=\sigma^*K_{G_0}$;

 \item If $p_g(\Gamma)=0$ and $\phi^*K_X-bG$ is $\mathbb{Q}$-effective for some positive rational number $b$, then 
$ \phi^*K_X|_G-\frac{b}{b+1}\sigma^*K_{G_0}$ is 
$\mathbb{Q}$-effective.

 \item If $p_g(X)\geq 2$ and $\Mov |\lfloor \phi^*K_X\rfloor|$ is base-point free and induces the fibration $g$, then 
\[
p_g(X) \phi^*K_X|_G\geq (p_g(X)-1) \sigma^*K_{G_0}.
\]
To be more precise, there exist effective divisors $D\in |K_X|$ and $D_0\in |K_{G_0}|$ such that 
\[
p_g(X) \phi^*D|_G-(p_g(X)-1) \sigma^*D_0
\]
is an effective $\mathbb{Q}$-divisor.
\end{enumerate}

\end{lemma}
\begin{proof}
(1) By \cite{hacon2007}, the fibers of $\phi$ are rationally chain connected, which implies that $g$ descends to a morphism $g_0: X\to \Gamma$.
As $K_X$ is nef, a general fiber of $g_0$ is a minimal projective surface, which is just the minimal model $G_0$ of $G$. Hence we have $K_X|_{G_0}=K_{G_0}$ by the adjunction, which gives the desired equality by pullback.

\medskip

By (1), Assertion (3) holds automatically when 
$p_g(\Gamma)>0$. 
From now on, we may assume that $p_g(\Gamma)=0$.
By taking higher models, it is clear that the assertions are independent of the choice of $\phi: X'\to X$.
So we may assume that 
$\phi$ is the birational projective morphism constructed in \cite[Proposition~3.1]{CHJ25} satisfying that 
\begin{align}
 \phi^*(K_X+\phi_*G)-K_{X'}-G=\widetilde{E}\geq 0\label{eq: K+G>0}
\end{align}
is an effective $\phi$-exceptional $\mathbb{Q}$-divisor.

\medskip

(2) By assumption, $\phi^*K_X-bG$ is $\mathbb{Q}$-effective. In particular, $K_X-b\phi_*G$ is $\mathbb{Q}$-effective. 
So by \eqref{eq: K+G>0}, 
$(1+\frac1b)\phi^*K_X-K_{X'}-G$
is $\mathbb{Q}$-effective.
Hence $ (1+\frac1b)\phi^*K_X|_G- \sigma^*K_{G_0}$ is 
$\mathbb{Q}$-effective by restricting on $G$.

\medskip

(3) 
Write $d=p_g(X)-1$, we have $K_X=d\phi_*G+\Delta$, where $\Delta$ is the fixed part of $|K_X|$. Take a general fiber $G'$ of $g$ different from $G$, 
then $D=d\phi_*G'+\Delta\in|K_X|$. Write $K_{X'}=\phi^*K_X+E_{\phi}$, then $\phi^*D+E_{\phi}\in |K_{X'}|$ and $(\phi^*D+E_{\phi})|_{G}\in |K_{G}|$.
Take $D_0=\sigma_*((\phi^*D+E_{\phi})|_{G})\in |K_{G_0}|$, then 
\begin{align}
(\phi^*D+E_{\phi})|_{G}-\sigma^*D_0=K_{G}-\sigma^*K_{G_0}\geq 0 \label{eq:phiD+E-D>0}
\end{align}
is an effective $\mathbb{Q}$-divisor. On the other hand, by \eqref{eq: K+G>0}, we have
\[
\phi^*(D+\phi_*G')=\phi^*D+E_{\phi}+G'+\widetilde{E}.
\]
Hence
\[
(d+1)\phi^*D\geq d\phi^*(D+\phi_*G')=d(\phi^*D+E_{\phi}+G'+\widetilde{E})\geq d(\phi^*D+E_{\phi}).
\]
Restricting both sides on $G$, we conclude by \eqref{eq:phiD+E-D>0} that 
\[
(d+1)\phi^*D|_{G}-d\sigma^*D_0\geq 0.
\]
This concludes the proof.
\end{proof}

\section{Noether-type inequality for threefolds of Kodaira dimension $2$}
In this section, we study the Noether-type inequality for $3$-folds of Kodaira dimension $2$. 
The discussion is divided into $2$ parts depending on whether $|K_X|$ is composed with a pencil of surfaces.

\subsection{Non-pencil case}
We first treat the case when $|K_X|$ is not composed with a pencil of surfaces.

\begin{theorem}\label{thm:Noehter not pencil}
Let $X$ be a minimal projective $3$-fold with $\kappa(X)=2$. Assume that $p_g(X)\geq 2$ and $|K_X|$ is not composed with a pencil. Then \[
\Ivol(X)\geq p_g(X)-2.
\]
The equality holds only if $|K_X|$ is base point free and induces a morphism \[\Phi_{|K_X|}: X\to T\subseteq \mathbb{P}^{p_g(X)-1}\] where $T$ is a surface of minimal degree $\deg T=p_g(X)-2$ and $\Phi_{|K_X|}$ coincides with the Iitaka fibration of $X$.

\end{theorem}
This inequality is optimal by Example~\ref{example: not a pencil}.
\begin{proof}
Keep the notation in Setting~\ref{setting} for $m=1$ with the subscript omitted. 
Then $T\subseteq \mathbb{P}^{p_g(X)-1}$ is a non-degenerate surface, which implies that $\mathcal{O}_T(1)^2\geq p_g(X)-2$ by 
\cite[Corollary~18.12]{Harris1992AlgebraicGeometryFirstCourse}. Therefore, 
\[
\Ivol(X)=L^2\geq L'^2\geq \mathcal{O}_T(1)^2\geq p_g(X)-2,
\]
where the first inequality follows from Lemma~\ref{lem:compare L and L',m}.

The equality $L^2=L'^2$ implies that $\phi'^*L\sim_\bQ L'$ by Lemma~\ref{lem:compare L and L',m} and the Hodge index theorem (see \cite[Lemma~2.2]{CHJ25}). Pulling back by $f'$, we have that
$\phi^*K_X\sim_\bQ f'^*L'$, 
where $f'^*L'\in \Mov|\lfloor \phi^*K_X\rfloor|$ by \eqref{eq:f*L=Mov}.
So we get $\phi^*K_X\sim f'^*L'$ and it is base point free. Hence we have a morphism $\Phi_{|K_X|}: X\to T\subseteq \mathbb{P}^{p_g(X)-1}$ where $\deg T=p_g(X)-2$. Here $T$ is a non-degenerate surface of minimal degree. 
In particular, it is normal, see \cite[Exercise~IV.18.4]{Beauville}.
Moreover, the equality $L'^2=\mathcal{O}_{T}(1)^2$ implies that the generically finite morphism $Z'\to T$ is of degree $1$. Hence $\Phi_{|K_X|}$ is exactly the Iitaka fibration. 
\end{proof}

\subsection{Pencil case}
We now consider the case when $|K_X|$ is composed with a pencil of surfaces $G$. The following proposition gives an inequality relating $\Ivol(X)$ and $\Ivol(G)$.

\begin{proposition}\label{prop:Noehter pencil}
Let $X$ be a minimal projective $3$-fold with $\kappa(X)=2$ and let $m$ be a positive integer.
Assume that $h^0(X, mK_X)\geq 2$ and that $|mK_X|$ is composed with a pencil of surfaces $G$.
Then 
\[
\Ivol(X)\geq \frac{\mu^2}{\mu+1} \Ivol(G),
\]
where $\mu\coloneqq \frac{h^0(X, mK_X)-1}{m}$. 
\end{proposition}
\begin{proof}
Keep the notation in Setting~\ref{setting}. 
Recall we have the following commutative diagram:
\begin{equation*}
\begin{tikzcd}
X' \arrow[r, "f'"] \arrow[d, "\phi"'] \arrow[rr, "g_m", bend left, shift left] & Z' \arrow[d, "\phi'"'] \arrow[r, "g'_m"] 
& \Gamma_m \arrow[d, "s_m"] & \\
X \arrow[r, "f"] \arrow[rr, "\Phi_{|mK_X|}", dashed, bend right, shift right] & Z \arrow[r, dashed] & T_m \arrow[r, hook] & \mathbb{P}^{h^0(X, mK_X)-1}
\end{tikzcd}
\end{equation*}
Let $F$ be a general fiber of $g'_m$, which is a smooth projective curve by Bertini's theorem. Note that $f': X'\to Z'$ induces an elliptic fibration $f_G: G\to F$ between general fibers of $g_m$ and $g'_m$.
By \eqref{eq:f*L=Mov} and Remark~\ref{rem:pencil}, we have 
$M_m \equiv aG$ and $L'_m\equiv a F$ with \begin{align}
 a\geq h^0(X, mK_X)-1.\label{eq: a>p-1}
\end{align}

Since $\phi'^*L$ and $L'_m$ are nef and $m\phi'^*L-L'_m$ is $\mathbb{Q}$-effective by Lemma~\ref{lem:compare L and L',m}, we have 
\begin{align}
\Ivol(X)=L^2=(\phi'^*L)^2\geq \frac1{m}(\phi'^*L\cdot L'_m)\geq \frac am(\phi'^*L\cdot F).\label{eq:L2>kLF}
\end{align}

 Denote by $\sigma: G\to G_0$ the morphism to the minimal model of $G$, then the natural morphism $f_0:G_0\to F$ is just the Iitaka fibration of $G_0$ with $f_G=f_0\circ \sigma$. In particular, there exists a $\mathbb{Q}$-divisor $L_0$ on $F$ such that $K_{G_0}\sim_\mathbb{Q}f_0^*L_0$. By Lemma~\ref{lem:compare KX KG}(1)(2) with $b=\frac{a}{m}$, $\phi^*K_X|_G-\frac{a}{a+m}\sigma^*K_{G_0}$ is $\mathbb{Q}$-effective. Also we have 
$f_G^*(\phi'^*L|_F)=\phi^*K_X|_G$.
So $f_G^*(\phi'^*L|_F)-\frac{a}{a+m}\sigma^*f_0^*L_0$ is $\mathbb{Q}$-effective, which implies that 
\begin{align}
\deg(\phi'^*L|_F)\geq \frac{a}{a+m}\deg(L_0)=\frac{a}{a+m} \Ivol(G_0)= \frac{a}{a+m} \Ivol(G).\label{eq: deg L >= IvolG}
\end{align}
So we get the desired inequality by combining with \eqref{eq: a>p-1} and \eqref{eq:L2>kLF}.
\end{proof}
\begin{remark}\label{rem after 3.2}
In the proof of Proposition~\ref{prop:Noehter pencil}, we remark that the equality in \eqref{eq: deg L >= IvolG} holds only if $\deg(\phi'^*L|_F)= \frac{a}{a+m}\deg(L_0)$, which implies that
$\phi^*K_X|_G\sim_\bQ\frac{a}{a+m}\sigma^*K_{G_0}$. So the equality in Proposition~\ref{prop:Noehter pencil} holds only if $\phi^*K_X|_G\sim_\bQ\frac{h^0(X, mK_X)-1}{h^0(X, mK_X)-1+m}\sigma^*K_{G_0}$.
\end{remark}
In order to 
apply Proposition~\ref{prop:Noehter pencil} to obtain a Noether-type inequality, we need to study $\Ivol(G)$ where $G$ is a smooth projective surface with $\kappa(G)=1$ and $p_g(G)>0$. The following lemma characterizes such $G$ with small Iitaka volume. 

\begin{lemma}\label{lem: Ivol S<1}
 Let $G$ be a smooth projective surface with $\kappa(G)=1$. 
 Assume that $p_g(G)>0$ and $\Ivol(G)<1$. Then the following hold:
 \begin{enumerate}
 \item $p_g(G)=1$, $h^1(G, \mathcal{O}_{G})=0$, $h^0(G, 2K_G)=2$, and $\Ivol(G)=1-\frac{1}{m_1}$ for an integer $m_1\geq 2$;
 
 \item the Iitaka fibration of $G_0$ has exactly one multiple fiber $m_1R$ with multiplicity $m_1$ and $(m_1-1)R\in |K_{G_0}|$ is the unique element, where $G_0$ is the minimal model of $G$. 
 \end{enumerate}
 
\end{lemma}

\begin{proof}
Without loss of generality, we may assume that $G$ is minimal. 
Denote by $f_G:G\to F$ the Iitaka fibration.
By Kodaira's canonical bundle formula \cite[Theorem~V.12.1]{BHPV2004}, we have 
\begin{align}
K_{G}\sim f_G^*\left(K_F+L_F+\sum_{i=1}^r\frac{m_i-1}{m_i}P_i\right),\label{eq:cbf elliptic fib}
\end{align}
where $L_F$ is a divisor on $F$ such that $\mathcal{O}_F(L_F)=(R^1{f_G}_*\mathcal{O}_{G})^{\vee}$, $P_i$ is a prime divisor on $F$ and $f_G^{-1}(P_i)$ is a multiple fiber of $f_G$ with multiplicity $m_i$ for each $1\leq i\leq r$. By the relative duality 
and \cite[Proposition~V.12.2, Theorem~III.18.2]{BHPV2004}, we have \[\chi(G, \mathcal{O}_{G})=\deg L_F=\deg {f_G}_*\omega_{G/F}\geq 0.\] In particular, 
\[
\Ivol(G)=2p_g(F)-2+\chi(G, \mathcal{O}_{G})+\sum_{i=1}^r\frac{m_i-1}{m_i}.
\]
Since $\Ivol(G)<1$, we have $p_g(F)\leq 1$.

Suppose that $p_g(F)=1$. Then we have $\chi(G, \mathcal{O}_G)=0$ and $r=1$ by $\Ivol(G)<1$. 
Note that $h^1(G, \mathcal{O}_{G})=1+p_g(G)-\chi(G, \mathcal{O}_G)\geq 2$, so by \cite[Lemma~1.6]{Serrano1991}, we have $h^1(G, \mathcal{O}_{G})=2=p_g(F)+1$. But we get a contradiction by \cite[Proposition~5.6]{Serrano1996} as $f_G$ cannot have exactly one multiple fiber, namely, $r\neq 1$.

Suppose that $p_g(F)=0$. 
Then by \cite[Lemma~1.6]{Serrano1991}, either $h^1(G, \mathcal{O}_{G})=p_g(F)=0$, or $\chi(G, \mathcal{O}_{G})=0$ and $h^1(G, \mathcal{O}_{G})=p_g(F)+1=1$.
In the latter case, 
\[0=\chi(G, \mathcal{O}_{G})=p_g(G)-h^1(G, \mathcal{O}_{G})+1=p_g(G)>0,\] which is absurd.
So we have $h^1(G, \mathcal{O}_{G})=p_g(F)=0$ and $\chi(G, \mathcal{O}_{G})=p_g(G)+1\geq 2$. Then $\Ivol(G)<1$ implies that \[\chi(G, \mathcal{O}_{G})=2, \quad \Ivol(G)=\frac{m_1-1}{m_1}, \quad r=1 .\]
In particular, $f_G$ has exactly one multiple fiber $m_1R$ with multiplicity $m_1$. Then by \eqref{eq:cbf elliptic fib}, $K_G\sim (m_1-1)R$ and $2K_G\sim (2m_1-2)R$, which implies that $p_g(G)=1$ and $h^0(G, 2K_G)=2$. 
\end{proof}

The inequality provided by Proposition~\ref{prop:Noehter pencil} is not good enough when $\Ivol(G)$ is small (e.g., $\Ivol(G)<\frac23$). To improve this inequality, we shall further consider the bicanonical system $|2K_X|$ of $X$. 
In the following we prove two key lemmas on $|2K_X|$ in this case: one estimates the lower bound of bicanonical genus, the other shows that the bicanonical map is actually not composed with a pencil.

\begin{lemma}
\label{lem:P2 >= 3pg-3}
Let $X$ be a minimal projective $3$-fold with $\kappa(X)=2$. 
Assume that $p_g(X)\geq 2$ and that $|K_X|$ is composed with a pencil of surfaces $G$ with $p_g(G)=1$ and $h^1(G, \mathcal{O}_{G})=0$. Then
\[
h^0(X, 2K_X)\geq 3p_g(X)-3.
\]
The equality holds only if $X$ is Gorenstein and $h^1(X, \mathcal{O}_{X})=h^2(X, \mathcal{O}_X)=0$.
\end{lemma}
\begin{proof}
First, we show that $h^2(X, 2K_X)=0$. 
The proof is well-known to experts (see \cite[Lemma~8.5]{Miyaoka1987}), but we reproduce it here for completeness. 
Take a general sufficiently ample Cartier divisor $S$ on $X$, we may assume that $S$ is a smooth projective surface by Bertini's theorem as $X$ has only isolated singularities by \cite[Corollary~5.18]{Kollar-Mori-1998}. 
By the exact sequence 
\[
 H^1(S,K_S+K_X|_S)\to H^2(X, 2K_X)\to H^2(X, 2K_X+S),
\]
it suffices to show that \[H^1(S,K_S+K_X|_S)=H^2(X, 2K_X+S)=0.\] 
It is clear that $H^2(X, 2K_X+S)=0$
by the Serre vanishing 
since $S$ is sufficiently ample. 
On the other hand,
denote by $f: X\to Z$ the Iitaka fibration of $X$ with $K_X\sim_\mathbb{Q}f^*L$, then the induced morphism $f_S: S\to Z$ is generically finite and 
$K_X|_S\sim_\mathbb{Q} f_S^*L$ is nef and big on $S$, which implies that $H^1(S,K_S+K_X|_S)=0$ by the Kawamata--Viehweg vanishing theorem \cite[Theorem~1-2-3]{KMM87}. Therefore, $h^2(X, 2K_X)=0$.
In particular, $h^0(X, 2K_X)\geq \chi(X, \mathcal{O}_X(2K_X)).$

As $K_X^3=0$, by Reid's Riemann--Roch formula \cite[Corollary~(10.3)(5)]{Reid-1987-YPG}, we have 
\begin{align}
\chi(X, \mathcal{O}_X(2K_X))\geq-3\chi(X, \mathcal{O}_X)=3\chi(X, \omega_X).\label{eq: 2K=3O}
\end{align}
So it suffices to show that $\chi(X, \omega_X)\geq p_g(X)-1$. 

Keep the notation in Setting~\ref{setting} for $m=1$ with the subscript omitted. Take $G$ to be a general fiber of $g: X'\to \Gamma$. 
By \cite[Theorem~2.1]{Kollar1986-I}, for each $i\geq 0$, $R^ig_*\omega_{X'}$ is torsion free, hence locally free since $\Gamma$ is a smooth curve. In particular, since $p_g(G)=1$ and $h^1(G, \mathcal{O}_{G})=0$, $g_*\omega_{X'}$ is a line bundle and $R^1g_*\omega_{X'}=0$ by \cite[Corollary~III.12.9]{Hartshorne1977}. 
Also we have 
$R^2g_*\omega_{X'}=\omega_{\Gamma}$ by the relative duality. 
Then by \cite[Corollary~3.2]{Kollar1986-II}, 
we have
\begin{align*}
\chi(X, \omega_X)=\chi(X', \omega_{X'})=& \chi(\Gamma, g_*\omega_{X'})-\chi(\Gamma, R^1g_*\omega_{X'})+\chi(\Gamma, R^2g_*\omega_{X'})\\
=& \chi(\Gamma, g_*\omega_{X'})+\chi(\Gamma, \omega_\Gamma)\\
=&\deg(g_*\omega_{X'})\\
\geq & p_g(X)-1.
\end{align*}
Here the last equality follows from the Riemann--Roch formula and the inequality follows from Lemma~\ref{lem:deg>=h0-1}. 

Finally, if the equality holds, then the equality in \eqref{eq: 2K=3O} holds, which implies that $X$ is Gorenstein by \cite[Corollary~(10.3)(5)]{Reid-1987-YPG} again. Also $\deg(g_*\omega_X)=p_g(X)-1$ implies that $p_g(\Gamma)=0$ by Lemma~\ref{lem:deg>=h0-1} and $h^1(X, \mathcal{O}_{X})=h^2(X, \mathcal{O}_X)$. 
A standard spectral sequence argument (see \cite[\S 2.6]{EXPII}) shows that
\[
h^1(X, \mathcal{O}_X)=p_g(\Gamma)+h^1(\Gamma, R^1g_*\omega_{X'})=0. \qedhere
\]
\end{proof}

\begin{lemma}\label{lem: 2K not pencil}
Let $X$ be a minimal projective $3$-fold with $\kappa(X)=2$. Assume that $p_g(X)\geq 3$ and $|K_X|$ is composed with a pencil of surfaces $G$ with $\Ivol(G)<\frac23$.
Then $|2K_X|$ is not composed with a pencil of surfaces.
\end{lemma}
\begin{proof}
Keep the notation in Setting~\ref{setting} for $m=1$ with the subscript omitted. Take $G$ to be a general fiber of $g: X'\to \Gamma$. 
By Remark~\ref{rem:pencil}, $\kappa(G)=1$. 
Also $p_g(X)>0$ implies that $p_g(G)>0$ by the adjunction formula $K_{X'}|_G=K_G$. 
So $G$ satisfies the conditions of Lemma~\ref{lem: Ivol S<1}. 

Write $p_g=p_g(X)$ for simplicity. By Remark~\ref{rem:pencil}, we may write 
\[\phi^*K_X\sim_{\bQ} \sum_{j=1}^{a}G_j+E\equiv aG+E\]
where $a\geq p_g-1$, each $G_j$ is a general fiber of $g$ with \[\sum_{j=1}^{a}G_j=M_1\in \Mov|\lfloor\phi^*K_X\rfloor|,\] and $E$ is an effective $\bQ$-divisor on $X'$. 
\begin{claim}\label{claim1}
 The $\bQ$-divisor 
\[
N\coloneq \phi^*K_X-\frac{1}{p_g-1}E-G
\]
is semiample and $mN-G$ is $\mathbb{Q}$-effective for a sufficiently large integer $m$. 
\end{claim}
\begin{proof}
 If $p_g(\Gamma)=0$, then by Remark~\ref{rem:pencil}, we have $a=p_g-1$ and that 
 \[N\sim_\mathbb{Q}\left(1-\frac1{p_g-1}\right)\phi^*K_X\]
 is a semiample $\mathbb{Q}$-divisor. 

 If $p_g(\Gamma)>0$, then by Remark~\ref{rem:pencil}, we have $a>p_g-1$ and
 \[N\sim_\mathbb{Q}\left(1-\frac1{p_g-1}\right)\phi^*K_X+\frac{1}{p_g-1}(M_1-(p_g-1)G)\]
is a semiample $\mathbb{Q}$-divisor, as $\phi^*K_X$ and $M_1-(p_g-1)G$ are both semiample. Here recall that $M_1-(p_g-1)G$ is the pull-back of an ample $\mathbb{Q}$-divisor on $\Gamma$. 

Finally, in either case, $mN-G$ is $\mathbb{Q}$-effective for a sufficiently large integer $m$ as $\phi^*K_X-G$ is $\mathbb{Q}$-effective.
\end{proof}
After replacing $X'$ by a higher model, we may assume that $N$ has simple normal crossing support. 
By Claim~\ref{claim1} and \cite[Theorem~6-1-6]{KMM87}, the natural map
\begin{align*}
& H^1\left(X',K_{X'}+\left\lceil \phi^*K_X-\frac1{p_g-1}E-G \right\rceil\right)\\\to &H^1\left(X',K_{X'}+\left\lceil \phi^*K_X-\frac1{p_g-1}E \right\rceil\right)
\end{align*}
is injective, which implies that the natural map
\begin{align}\label{eq:K+K-tE}
H^0\left(X',K_{X'}+\left\lceil \phi^*K_X-\frac1{p_g-1}E\right\rceil\right)\to H^0(G,D)
\end{align}
is surjective,
where
\[D\coloneqq \left(K_{X'}+\left\lceil \phi^*K_X-\frac1{p_g-1}E\right\rceil\right)|_G=K_G+\left\lceil\phi^*K_X|_G-\frac1{p_g-1}E|_G\right\rceil,\] here 
the restriction on $G$
commutes with the roundup 
since $G$ is general.

Let $\sigma:G\to G_0$ be the morphism to the minimal model. By Lemma~\ref{lem:compare KX KG}(3), 
we know that $\phi^*K_X|_G-\frac{p_g-1}{p_g}\sigma^*K_{G_0}$ is an effective $\mathbb{Q}$-divisor where we regard $\phi^*K_X$ and $\sigma^*K_{G_0}$ as well-chosen effective divisors as in Lemma~\ref{lem:compare KX KG}(3). 
Then
\begin{align*}
&\phi^*K_X|_G-\frac1{p_g-1}E|_G\geq \left(1-\frac1{p_g-1}\right)\phi^*K_X|_G\\
\geq &\left(1-\frac{1}{p_g-1}\right)\frac{p_g-1}{p_g}\sigma^*K_{G_0}= \left(1-\frac{2}{p_g}\right)\sigma^*K_{G_0}. 
\end{align*}
So we have
\begin{align}\label{eq D>2K+K}
D\geq K_G+\left\lceil \left(1-\frac{2}{p_g}\right)\sigma^*K_{G_0}\right\rceil=2\sigma^*K_{G_0}+K_{G/G_0}-\left\lfloor\frac{2}{p_g}\sigma^*K_{G_0}\right\rfloor.
\end{align}
By Lemma~\ref{lem: Ivol S<1}, $K_{G_0}\sim R$ where $2R$ is the unique multiple fiber of the Iitaka fibration of $G_0$. Here we used the condition $\Ivol(G)<\frac23$ to guarantee that $m_1=2$ in Lemma~\ref{lem: Ivol S<1}.
By Kodaira's classification of singular fibers \cite[V.7]{BHPV2004}, $R$ is of type $I_0$, $I_1$ or $I_b$ with $b\geq2$, which means that $R$ is an analytically normal crossing divisor.
Hence by \cite[Example~9.2.13]{lazarsfeld2004-II} and $p_g\geq 3$, we have 
\[
\sigma_*\mathcal{O}_G\left(K_{G/G_0}-\left\lfloor\frac{2}{p_g}\sigma^*R\right\rfloor\right)=\mathcal{J}\left(\frac{2}{p_g}R\right)=\mathcal{O}_{G_0}.
\]
Therefore, by the projection formula and Lemma~\ref{lem: Ivol S<1}, \eqref{eq D>2K+K} implies that 
\[
h^0(G, D)\geq h^0(G_0, 2K_{G_0})=2.
\]

Suppose that $|2K_X|$ is composed with a pencil of surfaces, then it is composed with the same pencil as $|K_X|$, which induces the same fibration $g:X'\to \Gamma$. But by the surjectivity of \eqref{eq:K+K-tE} and $h^0(G, D)\geq 2$, we know that $\Mov|\lfloor\phi^*(2K_X)\rfloor|$ induces a non-trivial movable linear system on $G$, a contradiction. 
\end{proof}

We are now prepared to prove the optimal Noether-type inequality in the case when $|K_X|$ is composed with a pencil of surfaces. 

\begin{theorem}\label{thm: noether pencil}
Let $X$ be a minimal projective $3$-fold with $\kappa(X)=2$. Assume that $p_g(X)\geq 2$ and $|K_X|$ is composed with a pencil of surfaces.
Then
\[
\Ivol(X)\geq \frac{3}{4}p_g(X)-\frac{5}{4}.
\]
\end{theorem}
This inequality is optimal 
by Example~\ref{example: pencil}.
\begin{proof}
Keep the notation in Setting~\ref{setting} for $m=1$ with the subscript omitted.
Take $G$ to be a general fiber of $g: X'\to \Gamma$. Then $p_g(G)>0$ by the adjunction formula $K_{X'}|_G=K_G$.

If $\Ivol(G)\geq 1$, then 
Proposition~\ref{prop:Noehter pencil} implies that \[\Ivol(X)\geq \frac{(p_g(X)-1)^2}{p_g(X)}> \frac{3}{4}p_g(X)-\frac{5}{4}.\]

From now on, suppose that $\Ivol(G)< 1$, then Lemma~\ref{lem: Ivol S<1} implies that $p_g(G)=1$ and $h^1(G, \mathcal{O}_{G})=0$.

If $p_g(X)=2$, then Proposition~\ref{prop:Noehter pencil} implies that $\Ivol(X)\geq \frac{1}{4}$ as $\Ivol(G)\geq \frac{1}{2}$ by Lemma~\ref{lem: Ivol S<1}.

If $p_g(X)\geq 3$ and $|2K_X|$ is composed with a pencil of surfaces, then it is composed with the same pencil of surfaces $G$ as $|K_X|$, then
by Lemma~\ref{lem: 2K not pencil}, we have $\Ivol(G)\geq \frac{2}{3}$. 
Then 
Proposition~\ref{prop:Noehter pencil} for $m=2$ and Lemma~\ref{lem:P2 >= 3pg-3} imply that \[\Ivol(X)\geq \frac{(\frac32p_g(X)-2)^2 }{\frac32p_g(X)-1}\cdot \frac{2}{3}> \frac{3}{4}p_g(X)-\frac{5}{4}.\]

If $|2K_X|$ is not composed with a pencil of surfaces,
then $\dim T_2=2$ under the notation of Setting~\ref{setting} for $m=2$. 
By Lemma~\ref{lem:compare L and L',m}, \cite[Corollary~18.12]{Harris1992AlgebraicGeometryFirstCourse}, and Lemma~\ref{lem:P2 >= 3pg-3}, we have
\begin{align}
4\Ivol(X)=4L^2\geq(L'_2)^2\geq \mathcal{O}_{T_2}(1)^2\geq h^0(X, 2K_X)-2\geq 3p_g(X)-5.\label{eq: last inequality}
\end{align}
This concludes the proof.
\end{proof}

\begin{proof}[Proof of Theorem~\ref{thm:main1}]
 We may replace $V$ by its minimal model $X$. 
 We may assume that $p_g(X)\geq 2$, otherwise the inequality is trivial. Then the theorem follows from Theorem~\ref{thm:Noehter not pencil} and Theorem~\ref{thm: noether pencil}.
\end{proof}

\subsection{Description of the structure of $X$ for the equality case}
In this subsection, we discuss the properties of minimal projective $3$-folds for which the equality of Theorem~\ref{thm:main1} holds.

\begin{proof}[Proof of Theorem~\ref{thm:equality case}]
 If $|K_X|$ is not composed with a pencil, then $p_g(X)\geq 3$ and 
 by Theorem~\ref{thm:Noehter not pencil}, \[
\frac{3}{4}p_g(X)-\frac{5}{4}= \Ivol(X)
\geq p_g(X)-2.
 \]
 Hence $p_g(X)=3$ and $\Ivol(X)=1= p_g(X)-2$. Then the conclusion follows from Theorem~\ref{thm:Noehter not pencil}.

 Now suppose that $|K_X|$ is composed with a pencil. Keep the notation in the proof of Theorem~\ref{thm: noether pencil}.
 
First we exclude the case that $p_g(X)=2$. Suppose that $p_g(X)=2$, then 
$\Ivol(X)=\frac14$ and by the proof of Theorem~\ref{thm: noether pencil}, $|K_X|$ is composed with a pencil of surfaces $G$ with $\Ivol(G)=\frac12.$
Recall that $h^0(X, 2K_X)\geq 2p_g(X)-1=3$.
If 
$|2K_X|$ is not composed with a pencil of surfaces, then $h^0(X, 2K_X)>3$, and \eqref{eq: last inequality} implies that $\Ivol(X)\geq \frac12$, a contradiction. So $|2K_X|$ is composed with a pencil of surfaces which is the same pencil as $|K_X|$, and Proposition~\ref{prop:Noehter pencil} for $m=2$ implies that $\mu=\frac{h^0(X, 2K_X)-1}{2}\geq 1$ and
\[
\frac14=\Ivol(X)\geq \frac{\mu^2}{\mu+1}\cdot \frac12\geq \frac 14.
\]
This implies that $h^0(X, 2K_X)= 3=3p_g(X)-3$. 
By Lemma~\ref{lem: Ivol S<1} and Lemma~\ref{lem:P2 >= 3pg-3}, $X$ is Gorenstein. 
Recall that the equality in Proposition~\ref{prop:Noehter pencil} for $m=1$ holds, so by Remark~\ref{rem after 3.2}, \begin{align}
 \phi^*K_X|_G\sim_\bQ\frac{1}{2}\sigma^*K_{G_0}.\label{eq: K=1/2KG}
\end{align} 
Here $\phi: X'\to X$ is a resolution resolving the pencil and $\sigma:G\to G_0$ is the morphism to the minimal model of $G$. 
By Lemma~\ref{lem: Ivol S<1}, the reduced fiber $R\in |K_{G_0}|$ is the unique element. So by Lemma~\ref{lem:compare KX KG}(3), there is an element $D\in |K_X|$ such that $2\phi^*D|_G\geq \sigma^*R$. 
By \eqref{eq: K=1/2KG}, we have 
$2\phi^*D|_G= \sigma^*R$. But this is absurd as $R$ is reduced and $D$ is Cartier.

 Now we consider $p_g(X)\geq 3$. Then
 we know that the equality holds only if the following properties are satisfied:
 \begin{itemize}
 \item $|K_X|$ is composed with a pencil of surfaces $G$ with $\Ivol(G)<1$, $p_g(G)=1$, and $h^1(G, \mathcal{O}_{G})=0$;
 \item $|2K_X|$ is not composed with a pencil of surfaces. 
 \end{itemize}
Then by \eqref{eq: last inequality}, we have
\[4L^2=(L_2')^2=\mathcal{O}_{T_2}(1)^2=h^0(X, 2K_X)-2=3p_g(X)-5.\]

The equality $h^0(X, 2K_X)=3p_g(X)-3$ implies that $X$ is Gorenstein and $h^1(X, \mathcal{O}_{X})=h^2(X, \mathcal{O}_X)=0$ by Lemma~\ref{lem:P2 >= 3pg-3}.

The equality $4L^2=(L_2')^2$ implies that $2\phi'^*L\sim_\bQ L_2'$ by Lemma~\ref{lem:compare L and L',m} and the Hodge index theorem (see \cite[Lemma~2.2]{CHJ25}). Pulling back by $f'$, we have that
$\phi^*(2K_X)\sim_\bQ f'^*L_2'$, 
where $f'^*L_2'\in \Mov|\lfloor \phi^*(2K_X)\rfloor|$ by \eqref{eq:f*L=Mov}.
So we get $\phi^*(2K_X)\sim f'^*L_2'$ and it is base point free. Hence we have a morphism $\Phi_{|2K_X|}: X\to T_2\subseteq \mathbb{P}^{3p_g(X)-4}$ where $\deg T_2=3p_g(X)-5$. Here $T_2$ is a non-degenerate surface of minimal degree. 
In particular, it is normal, see \cite[Exercise~IV.18.4]{Beauville}.
Moreover, the equality $(L_2')^2=\mathcal{O}_{T_2}(1)^2$ implies that the generically finite morphism $Z'\to T_2$ is of degree $1$. Hence $\Phi_{|2K_X|}$ is exactly the Iitaka fibration. 

Finally we determine $\Ivol(G)$. 
Let $\sigma:G\to G_0$ be the morphism to the minimal model. 
By Lemma~\ref{lem: Ivol S<1}, $\Ivol(G)=1-\frac{1}{m_1}$ for some integer $m_1\geq 2$ and the Iitaka fibration of $G_0$ has exactly one multiple fiber $m_1R$ with multiplicity $m_1$ and $(m_1-1)R\in |K_{G_0}|$ is the unique element. 
Since $2K_X$ is free and defines the Iitaka fibration of $X$, $\phi^*(2K_X)|_G$ is free and defines the Iitaka fibration of $G$, which implies that $\phi^*(2K_X)|_G\sim \sigma^*(km_1R)$ for some positive integer $k$. On the other hand, by the Zariski decomposition of $2K_G$, we know that 
 $\phi^*(2K_X)|_G\leq \sigma^*(2K_{G_0})\sim \sigma^*((2m_1-2)R)$. So we conclude that $k=1$ and $\phi^*(2K_X)|_G\sim \sigma^*(m_1R)$.
Fix an element $D\in |K_X|$, then we have
$2\sigma_*(\phi^*D|_G)\sim m_1R$. Since $X$ is Gorenstein, $\sigma_*(\phi^*D|_G)$ is a Weil divisor. Since $m_1R$ is the unique multiple fiber of $G_0$, we have 
$2\sigma_*(\phi^*D|_G)= m_1R$ and in particular $m_1$ is even. 
On the other hand, 
by Lemma~\ref{lem:compare KX KG}, 
we know that \[\phi^*K_X|_G-\frac{p_g-1}{p_g}\sigma^*K_{G_0}\sim_\bQ \left(\frac{m_1}{2}-\frac{p_g-1}{p_g}(m_1-1)\right)\sigma^*R\] is $\mathbb{Q}$-effective where $p_g\coloneq p_g(X)$, which implies that
\[
\frac{m_1}{2}-\frac{p_g-1}{p_g}(m_1-1)\geq 0.
\]
Hence we get 
\[
m_1= \begin{cases} 2 \text{ or }4 &\text{if } p_g(X)=3;\\
2 &\text{if } p_g(X)\geq 4.
\end{cases}
\]
\end{proof}
\section{Examples}

\subsection{Examples of $3$-folds of Kodaira dimension $2$ obtained from (bi)double covers}\ 

In this subsection, we construct examples of minimal projective $3$-folds of Kodaira dimension $2$, 
which illustrate the optimality of our Noether-type inequality.

For a non-negative integer $e$, let $\mathbb{F}_e=\mathbb{P}_{\mathbb{P}^1}(\mathcal{O}_{\mathbb{P}^1}\oplus \mathcal{O}_{\mathbb{P}^1}(-e))$ be the $e$-th Hirzebruch surface. 
Let $\ell$ be the fiber class of $\mathbb{F}_e$ and let $\sigma_0$ be the tautological class on $\mathbb{F}_e$ such that 
$\sigma_0^2=-e$. 
We have $K_{\mathbb{F}_e}=-2\sigma_0-(e+2)\ell$. We will freely use facts about Hirzebruch surfaces in \cite[\S V.2]{Hartshorne1977}.

First, we construct smooth minimal projective $3$-folds $X$ with $\kappa(X)=2$ such that $\Ivol(X)=p_g(X)-2$ and $|K_X|$ is not composed with a pencil of surfaces, where $p_g(X)$ can take any integer in $\mathbb{Z}_{\geq 3}$.

\begin{example}\label{example: not a pencil}
Fix nonnegative integers $e,k$ such that $k\geq 3e,k\geq e+3$ and $e\geq 1$. 
Consider $Y\coloneqq \mathbb{P}^1\times \mathbb{F}_e$ with the natural projections $\pi_1$ and $\pi_2$. 
Take $L=\pi_2^*(3\sigma_0+k\ell)-\pi_1^*K_{\mathbb{P}^1}$ on $Y$ which is base-point free. 
 Let $\psi:X\to Y$ be the double cover branched over a general member $B\in|2L|$. 
We have 
\[K_X=\psi^*(K_Y+L)=\psi^*\pi_2^*(\sigma_0+(k-e-2)\ell).\] 
Hence $X$ is a smooth minimal projective $3$-fold with $\kappa(X)=2$ whose Iitaka fibration is just $\pi_2\circ \psi: X\to \mathbb{F}_e$.
Then $|K_X|$ is not composed with a pencil of surfaces and
we have $\Ivol(X)=2k-3e-4$ and $p_g(X)=2k-3e-2$. In particular, $\Ivol(X)=p_g(X)-2$ and $p_g(X)$ can take any integer in $\mathbb{Z}_{\geq 3}$. 
\end{example}

Then, we construct smooth minimal projective $3$-folds $X$ with $\kappa(X)=2$ such that $\Ivol(X)=\frac{3}{4}p_g(X)-\frac{5}{4}$ and $|K_X|$ is composed with a pencil of surfaces, where $p_g(X)$ can take any odd integer in $\mathbb{Z}_{\geq 3}$.

\begin{example}\label{example: pencil}
Fix an integer $e{\geq 2}$. 
Consider $Y\coloneqq \mathbb{P}^1\times \mathbb{F}_{2e}$ with natural projections $\pi_1$ and $\pi_2$. We denote the divisor class of $\pi_1^*\mathcal{O}_{\mathbb{P}^1}(a)\otimes \pi_2^*\mathcal{O}_{\mathbb{F}_{2e}}(b\sigma_0+c\ell)$ by $(a,b\sigma_0+c\ell)$ for simplicity. For example, $K_Y=(-2,-2\sigma_0-(2e+2)\ell)$.

Take $3$ divisor classes 
\begin{align*}
 &L_1=(1,\sigma_0+e\ell);\\
 &L_2=(1,2\sigma_0+3e\ell);\\
 &L_3=(2,2\sigma_0+4e\ell)
\end{align*}
on $Y$. 
Take general elements 
\begin{align*}
 & D_1\in |L_2+L_3-L_1|=|(2,3\sigma_0+6e\ell)|;\\ &D_2\in |L_1+L_3-L_2|=|(2,\sigma_0+2e\ell)|;\\
 &D_3\in |L_1+L_2-L_3|=|(0,\sigma_0)|.
\end{align*}
Here the first $2$ linear systems are base-point free and $D_3$ is smooth. Hence $D_1+D_2+D_3$ is a simple normal crossing divisor. Moreover, \[(D_1\cdot D_2\cdot D_3)=8((\sigma_0+2e\ell)\cdot \sigma_0)=0,\] hence $D_1\cap D_2\cap D_3=\emptyset$.

By \cite[Theorem~2]{Catanese1999}, there is a finite flat $(\mathbb{Z}/2\mathbb{Z})^2$-Galois cover $\psi:X\to Y$ such that $X$ is smooth. Moreover, by \cite[Remark~4]{Catanese1999}, we have \[2K_X=\psi^*(2K_Y+L_1+L_2+L_3)=\psi^*((0,\sigma_0+(4e-4)\ell)).\]
Hence $X$ is a smooth minimal projective $3$-fold with 
$\kappa(X)=2$, whose Iitaka fibration is just $\pi_2\circ \psi: X\to \mathbb{F}_{2e}$. We have
\[\Ivol(X)=\frac{1}{4}(\sigma_0+(4e-4)\ell)^2=\frac{3}{2}e-2.\]
By \cite[p.~102]{Catanese1999}, we have \[p_g(X)=p_g(Y)+\sum_{i=1}^3h^0(Y,K_Y+ L_i)=2e-1\] 
and the canonical map is $X\to \mathbb{F}_{2e}\to \mathbb{P}^1$ where the second morphism is induced by $(2e-2)\ell$. 
In particular, $\Ivol(X)=\frac{3}{4}p_g(X)-\frac{5}{4}$.
\end{example}

Since we are interested in the geography of minimal projective $3$-folds $X$ with $\kappa(X)=2$, it is natural to ask whether $\Ivol(X)$ can be bounded from above by a function of $p_g(X)$. The answer is simply no, by the following easy construction.

\begin{proposition}\label{prop: non-existence of reversed inequality}
For any integer $p\geq 2$, there is an infinite sequence of smooth minimal projective $3$-folds $\{X_i\}$ with $\kappa(X_i)=2$ and $p_g(X_i)=p$ such that $\lim_{i\to \infty}\Ivol(X_i)=\infty$. 
\end{proposition}
\begin{proof}
By \cite[Theorem~4.16]{CHL2025}, there exists a sequence of minimal projective surfaces $\{S_i\}$ such that $p_g(S_i)=1$, $\kappa(S_i)=1$, and $\lim_{i\to \infty}\Ivol(S_i)=\infty$. 
Let $C$ be a smooth projective curve of genus $p_g(C)=p$ and let $X_i\coloneqq C\times S_i$. 
Then $X_i$ is a smooth minimal projective $3$-fold with $\kappa(X_i)=2$.
By the K\"unneth formula, we have $p_g(X_i)=p$ and
\[
\Ivol(X_i)=\limsup_{m\to \infty} \frac{h^0(C,\omega_C^{\otimes m})\cdot h^0(S_i,\omega_{S_i}^{\otimes m})}{m^2/2}=4(p-1) \Ivol(S_i).
\]
This gives the desired sequence of examples. 
\end{proof}

\subsection{Examples of $3$-folds of Kodaira dimension $2$ obtained from weighted blowups} \ 

In this subsection, 
we give a formula for the Iitaka volumes of minimal projective $3$-folds of Kodaira dimension $2$ constructed in \cite[Lemma~4.4, Theorem~4.5]{CJL2024}. Those are infinite series of examples constructed from weighted blowups of weighted hypersurfaces. The following proposition applies to all examples in \cite[Lemma~4.4, Theorem~4.5]{CJL2024}.

\begin{proposition}\label{prop: CJL Y to Z}
Let $X\subseteq \mathbb{P}(b_1,\cdots,b_{5})$ be a $3$-dimensional well-formed quasismooth general weighted hypersurface of degree $d$. 
Denote by $x_1,\cdots,x_{5}$ the homogeneous coordinates of $\mathbb{P}(b_1,\cdots,b_{5})$. Suppose that $X$ has a cyclic quotient singularity at the point $Q=(x_1=x_2=x_3=0)\cap X$ of type $\frac{1}{r}(b_1, b_2, e_3)$ where $x_1,x_2,x_3$ are the local coordinates of $Q$ with weights $\frac{b_1}{r},\frac{b_2}{r},\frac{e_3}{r}$ respectively. 
Let $\pi:Y\to X$ be the weighted blowup of $X$ at $Q$ with weight $(b_1,b_2,e_3)$. 
Suppose that the following conditions hold:
\begin{enumerate}
 \item $b_1, b_2, e_3$ are coprime to each other;
 \item $e_3=b_3-ub_4-vb_5$ for some nonnegative integers $u,v$;
 \item $dre_3=b_3b_4b_5$;
 \item $ d-\sum_{i=1}^{5}b_i=r-b_1-b_2-e_3>0$.
\end{enumerate}
Then there is a fibration $f:Y\to \mathbb{P}(b_1,b_2,e_3)$ with \[K_Y\sim_\mathbb{Q}\left(d-\sum_{i=1}^{5}b_i\right)f^*(H),\]
where $H$ is the divisor class corresponding to $\mathcal{O}_{\mathbb{P}(b_1,b_2,e_3)}(1)$.

In particular, if $Y$ has terminal singularities, then $Y$ is a minimal projective $3$-fold with $\kappa(Y)=2$ and $\Ivol(Y)=\frac{(d-\sum_{i=1}^{5}b_i)^2}{b_1b_2e_3}$.
\end{proposition}

\begin{proof}
Consider the following dominant rational map 
\begin{align*}
\mathbb{P}(b_1,b_2,b_3,b_4,b_5)&\dashrightarrow \mathbb{P}(b_1,b_2,e_3),\\
[x_1:x_2:x_3:x_4:x_5]&\mapsto [x_1:x_2:x_3{x_4^{-u}x_5^{-v}}].
\end{align*}
Let $p:X\dashrightarrow \mathbb{P}(b_1,b_2,e_3)$ be the induced map. Then $p$ is well-defined outside 
$Q=(x_1=x_2=x_3=0)\cap X$. 
The weighted blowup resolves the indeterminacy of $p$ and induces a fibration $f:Y\to \mathbb{P}(b_1,b_2,e_3)\eqqcolon
 Z$.
Denote by $E\simeq \mathbb{P}(b_1, b_2, e_3)$ the exceptional divisor of $\pi$, then $\mathcal{O}_Y(E)|_E=\mathcal{O}_E(-r)$ by \cite[Proposition~2.9]{CJL2024}. 
Take $H_X$ and $H_Z$ to be the divisor classes corresponding to $\mathcal{O}_X(1)$ and $\mathcal{O}_Z(1)$. 
Take $D=(x_1=0)\cap X\in |b_1H_X|$ and $D_Z=(x_1=0)\cap Z\in |b_1H_Z|$, then by the construction of $p$, we have $\pi^*D=f^*D_Z+tE$ for some rational number $t$. 
Note that 
\[
0=(f^*D_Z)^3=(\pi^*D-tE)^3=D^3-t^3E^3=\frac{b_1^3d}{b_1b_2b_3b_4b_5}- \frac{t^3r^2}{b_1b_2e_3}.
\]
Hence $t=\frac{b_1}{r}\sqrt[3]{\frac{dre_3}{b_3b_4b_5}}=\frac{b_1}{r}$.
This implies that
\[\pi^*H_X\sim_\mathbb{Q}f^*H_Z+\frac{1}{r}E.\] 
Then by \cite[Proposition~2.9]{CJL2024},
\begin{align*}
K_Y=&\pi^*K_X-\frac{r-b_1-b_2-e_3}{r}E\\
\sim_\mathbb{Q}&\left(d-\sum_{i=1}^{5}b_i\right)\pi^* H_X-\frac{r-b_1-b_2-e_3}{r}E\\
\sim_\mathbb{Q} & \left(d-\sum_{i=1}^{5}b_i\right)\left(\pi^* H_X-\frac{1}{r}E\right)\\
\sim_\mathbb{Q} &\left(d-\sum_{i=1}^{5}b_i\right)f^*H_Z.
\end{align*}
The last statement follows from the fact that $H_Z$ is ample and $H_Z^2=\frac{1}{b_1b_2e_3}$.
\end{proof}

\subsection{Examples of varieties of Kodaira dimension $1$}\ 

The following examples of smooth varieties of Kodaira dimension $1$ illustrate the optimality of our Noether-type inequality. 
\begin{example}\label{example: kodaira dim 1}
Fix an integer $n\geq 2$.
Consider $Y\coloneqq \mathbb{P}^1\times \mathbb{P}^{n-1}$ with natural projections $\pi_1$ and $\pi_2$.
Fix an integer $m\geq 3$ and take a base-point free divisor $L\coloneqq m\pi_1^*P-\pi_2^*K_{\mathbb{P}^{n-1}}$, where $P\in \mathbb{P}^1$ is a point. Let $\psi:X\to Y$ be the double cover branched over a general member $B\in|2L|$. 
We have \[K_X=\psi^*(K_Y+L)=(m-2)\psi^*\pi_1^*P.\] 
Hence $X$ is a smooth minimal projective variety of dimension $n$ with $\kappa(X)=1$ whose Iitaka fibration is just $\pi_1\circ \psi: X\to \mathbb{P}^1$.
We know that $p_g(X)=m-1$ and $\Ivol(X)=m-2$. 
\end{example}

Similar to Proposition~\ref{prop: non-existence of reversed inequality}, we have the following:

\begin{proposition} 
For integers $p\geq 0$ and $n\geq 2$, there is an infinite sequence of smooth minimal projective $n$-folds $\{X_i\}$ with $\kappa(X_i)=1$ and $p_g(X_i)=p$ such that $\lim_{i\to \infty}\Ivol(X_i)=\infty$. 
\end{proposition}
\begin{proof}
By \cite[Theorem~4.16]{CHL2025}, there exists a sequence of minimal projective surfaces $\{S_i\}$ such that $p_g(S_i)=p$, $\kappa(S_i)=1$, and $\lim_{i\to \infty}\Ivol(S_i)=\infty$. So these provide examples for $n=2$. If $n>2$,
let $A$ be an Abelian variety of dimension $n-2$ and let $X_i\coloneqq A\times S_i$. 
Then $X_i$ is a smooth minimal projective $n$-fold with $\kappa(X_i)=1$.
By the K\"unneth formula, we have $p_g(X_i)=p$ and
\[
\Ivol(X_i)=\limsup_{m\to \infty} \frac{h^0(A,\omega_A^{\otimes m})\cdot h^0(S_i,\omega_{S_i}^{\otimes m})}{m}= \Ivol(S_i).
\]
This gives the desired sequence of examples. 
\end{proof}

Naturally, out of curiosity about the statement of Theorem~\ref{thm:main1}, we would like to pose the following question: 
\begin{op}
Do optimal Noether-type inequalities exist for higher dimensional non-general type varieties of Kodaira dimension $\kappa>1$?
\end{op}

\section*{Acknowledgments}
We would like to thank Jungkai Chen and Tong Zhang for their interest in this problem. 
 
 This work was supported by National Key Research and Development Program of China \#2023YFA1010600 and NSFC for Innovative Research Groups \#12121001. 
 Hu was supported by NSFC grant \#12571044 and NSFC for Young Scientists Fund (B) \#12622102. 
 Chen and Jiang are members of the Key Laboratory of Mathematics for Nonlinear Sciences, Fudan University.

The main results of this paper was obtained around April 2026. All mathematical results in this paper were achieved by the authors without any assistance of AI tools. 


\bibliographystyle{amsplain}

\bibliography{reference}

@book{Hartshorne1977,
    AUTHOR = {Hartshorne, Robin},
     TITLE = {Algebraic geometry},
    SERIES = {Graduate Texts in Mathematics},
    VOLUME = {52},
 PUBLISHER = {Springer-Verlag, New York-Heidelberg},
      YEAR = {1977},
     PAGES = {xvi+496},
      ISBN = {0-387-90244-9},
   MRCLASS = {14-01},
  MRNUMBER = {463157},
MRREVIEWER = {Robert\ Speiser},
}

@book{lazarsfeld2004-II,
AUTHOR = {Lazarsfeld, Robert},
     TITLE = {Positivity in algebraic geometry. {II}},
    SERIES = {Ergebnisse der Mathematik und ihrer Grenzgebiete. 3. Folge. A
              Series of Modern Surveys in Mathematics [Results in
              Mathematics and Related Areas. 3rd Series. A Series of Modern
              Surveys in Mathematics]},
    VOLUME = {49},
      NOTE = {Positivity for vector bundles, and multiplier ideals},
 PUBLISHER = {Springer-Verlag, Berlin},
      YEAR = {2004},
     PAGES = {xviii+385},
      ISBN = {3-540-22534-X},
   MRCLASS = {14-02 (14C20 14F05 14F17)},
  MRNUMBER = {2095472},
MRREVIEWER = {Mihnea\ Popa},
       DOI = {10.1007/978-3-642-18808-4},
       URL = {https://doi.org/10.1007/978-3-642-18808-4},
}

@book{Kawamata_2024,  
    AUTHOR = {Kawamata, Yujiro},
     TITLE = {Algebraic varieties: minimal models and finite generation},
    SERIES = {Cambridge Studies in Advanced Mathematics},
    VOLUME = {214},
      NOTE = {Translated from the 2014 Japanese original by Chen Jiang},
 PUBLISHER = {Cambridge University Press, Cambridge},
      YEAR = {2024},
     PAGES = {x+251},
      ISBN = {978-1-009-34467-8; [9781009344647]},
   MRCLASS = {14E30},
  MRNUMBER = {4749084},
MRREVIEWER = {Izzet\ Coskun},
}

@article{hacon2007,
  AUTHOR = {Hacon, Christopher D. and M\textsuperscript{c}Kernan, James},
     TITLE = {On {S}hokurov's rational connectedness conjecture},
   JOURNAL = {Duke Math. J.},
  FJOURNAL = {Duke Mathematical Journal},
    VOLUME = {138},
      YEAR = {2007},
    NUMBER = {1},
     PAGES = {119--136},
      ISSN = {0012-7094,1547-7398},
   MRCLASS = {14E30 (14E05 14J45)},
  MRNUMBER = {2309156},
MRREVIEWER = {Mihnea\ Popa},
       DOI = {10.1215/S0012-7094-07-13813-4},
       URL = {https://doi.org/10.1215/S0012-7094-07-13813-4},
}

@book{Kollar-Mori-1998,
    AUTHOR = {Koll\'ar, J\'anos and Mori, Shigefumi},
     TITLE = {Birational geometry of algebraic varieties},
    SERIES = {Cambridge Tracts in Mathematics},
    VOLUME = {134},
      NOTE = {With the collaboration of C. H. Clemens and A. Corti,
              Translated from the 1998 Japanese original},
 PUBLISHER = {Cambridge University Press, Cambridge},
      YEAR = {1998},
     PAGES = {viii+254},
      ISBN = {0-521-63277-3},
   MRCLASS = {14E30},
  MRNUMBER = {1658959},
MRREVIEWER = {Mark\ Gross},
       DOI = {10.1017/CBO9780511662560},
       URL = {https://doi.org/10.1017/CBO9780511662560},
}

@book{Harris1992AlgebraicGeometryFirstCourse,
    AUTHOR = {Harris, Joe},
     TITLE = {Algebraic geometry},
    SERIES = {Graduate Texts in Mathematics},
    VOLUME = {133},
      NOTE = {A first course},
 PUBLISHER = {Springer-Verlag, New York},
      YEAR = {1992},
     PAGES = {xx+328},
      ISBN = {0-387-97716-3},
   MRCLASS = {14-01},
  MRNUMBER = {1182558},
MRREVIEWER = {Liam\ O'Carroll},
       DOI = {10.1007/978-1-4757-2189-8},
       URL = {https://doi.org/10.1007/978-1-4757-2189-8},
}

@article {N1870,
    AUTHOR = {Noether, Max},
     TITLE = {Zur {T}heorie des eindeutigen {E}ntsprechens algebraischer
              {G}ebilde von beliebig vielen {D}imensionen},
   JOURNAL = {Math. Ann.},
  FJOURNAL = {Mathematische Annalen},
    VOLUME = {2},
      YEAR = {1870},
    NUMBER = {2},
     PAGES = {293--316},
      ISSN = {0025-5831,1432-1807},
   MRCLASS = {99-04},
  MRNUMBER = {1509663},
       DOI = {10.1007/BF01444023},
       URL = {https://doi.org/10.1007/BF01444023},
}

@article{CCJ1,
    AUTHOR = {Chen, Jungkai A. and Chen, Meng and Jiang, Chen},
     TITLE = {The {N}oether inequality for algebraic $3$-folds},
      NOTE = {With an appendix by J\'anos Koll\'ar},
   JOURNAL = {Duke Math. J.},
  FJOURNAL = {Duke Mathematical Journal},
    VOLUME = {169},
      YEAR = {2020},
    NUMBER = {9},
     PAGES = {1603--1645},
      ISSN = {0012-7094,1547-7398},
   MRCLASS = {14E05 (14J30)},
  MRNUMBER = {4105534},
MRREVIEWER = {Anne-Sophie\ Kaloghiros},
       DOI = {10.1215/00127094-2019-0080},
       URL = {https://doi.org/10.1215/00127094-2019-0080},
}

@article{CCJ2,
    AUTHOR = {Chen, Jungkai A. and Chen, Meng and Jiang, Chen},
     TITLE = {\phantom{}{A}ddendum to “{T}he {N}oether inequality for algebraic $3$-folds”},
   JOURNAL = {Duke Math. J.},
  FJOURNAL = {Duke Mathematical Journal},
    VOLUME = {169},
      YEAR = {2020},
    NUMBER = {11},
     PAGES = {2199–2204},
      ISSN = {0012-7094},
   MRCLASS = {14E05 (14J30)},
  MRNUMBER = {4132583},
MRREVIEWER = {Anne-Sophie\ Kaloghiros},
       DOI = {10.1215/00127094-2020-0042},
       URL = {http://dx.doi.org/10.1215/00127094-2020-0042}
}

@article{CHJ25,
AUTHOR = {Chen, Meng and Hu, Yong and Jiang, Chen},
     TITLE = {The {N}oether inequality for $3$-folds and three moduli spaces
              with minimal volumes},
   JOURNAL = {Proc. Lond. Math. Soc. (3)},
  FJOURNAL = {Proceedings of the London Mathematical Society. Third Series},
    VOLUME = {131},
      YEAR = {2025},
    NUMBER = {2},
     PAGES = {Paper No. e70078, 41 pp},
      ISSN = {0024-6115,1460-244X},
   MRCLASS = {14J30 (14D22 14J10)},
  MRNUMBER = {4942287},
MRREVIEWER = {Caryn\ Werner},
       DOI = {10.1112/plms.70078},
       URL = {https://doi.org/10.1112/plms.70078},
}

@book{BHPV2004,
   AUTHOR = {Barth, Wolf P. and Hulek, Klaus and Peters, Chris A. M. and
              Van de Ven, Antonius},
     TITLE = {Compact complex surfaces},
    SERIES = {Ergebnisse der Mathematik und ihrer Grenzgebiete. 3. Folge. A
              Series of Modern Surveys in Mathematics [Results in
              Mathematics and Related Areas. 3rd Series. A Series of Modern
              Surveys in Mathematics]},
    VOLUME = {4},
   EDITION = {Second},
 PUBLISHER = {Springer-Verlag, Berlin},
      YEAR = {2004},
     PAGES = {xii+436},
      ISBN = {3-540-00832-2},
   MRCLASS = {14Jxx (14-02 32-02 32J15 57R57)},
  MRNUMBER = {2030225},
MRREVIEWER = {I.\ Dolgachev},
       DOI = {10.1007/978-3-642-57739-0},
       URL = {https://doi.org/10.1007/978-3-642-57739-0},
}

@article{CHL2025,
  AUTHOR = {Chen, Guodu and Han, Jingjun and Liu, Wenfei},
     TITLE = {On the {I}itaka volumes of log canonical surfaces and
              threefolds},
   JOURNAL = {J. Lond. Math. Soc. (2)},
  FJOURNAL = {Journal of the London Mathematical Society. Second Series},
    VOLUME = {111},
      YEAR = {2025},
    NUMBER = {4},
     PAGES = {Paper No. e70132, 45 pp},
      ISSN = {0024-6107,1469-7750},
   MRCLASS = {14E30 (14B05 14J27)},
  MRNUMBER = {4885118},
       DOI = {10.1112/jlms.70132},
       URL = {https://doi.org/10.1112/jlms.70132},
}

@article{Serrano1991,
  AUTHOR = {Serrano, Fernando},
     TITLE = {The {P}icard group of a quasi-bundle},
   JOURNAL = {Manuscripta Math.},
  FJOURNAL = {Manuscripta Mathematica},
    VOLUME = {73},
      YEAR = {1991},
    NUMBER = {1},
     PAGES = {63--82},
      ISSN = {0025-2611,1432-1785},
   MRCLASS = {14C22},
  MRNUMBER = {1124311},
       DOI = {10.1007/BF02567629},
       URL = {https://doi.org/10.1007/BF02567629},
}

@article{Serrano1996,
  AUTHOR = {Serrano, Fernando},
     TITLE = {Isotrivial fibred surfaces},
   JOURNAL = {Ann. Mat. Pura Appl. (4)},
  FJOURNAL = {Annali di Matematica Pura ed Applicata. Serie Quarta},
    VOLUME = {171},
      YEAR = {1996},
     PAGES = {63--81},
      ISSN = {0003-4622},
   MRCLASS = {14J25},
  MRNUMBER = {1441865},
MRREVIEWER = {Jean\ D'Almeida},
       DOI = {10.1007/BF01759382},
       URL = {https://doi.org/10.1007/BF01759382},
}

@article{CJL2024,
   AUTHOR = {Chen, Meng and Jiang, Chen and Li, Binru},
     TITLE = {On minimal varieties growing from quasismooth weighted
              hypersurfaces},
   JOURNAL = {J. Differential Geom.},
  FJOURNAL = {Journal of Differential Geometry},
    VOLUME = {127},
      YEAR = {2024},
    NUMBER = {1},
     PAGES = {35--76},
      ISSN = {0022-040X,1945-743X},
   MRCLASS = {14E30},
  MRNUMBER = {4753498},
MRREVIEWER = {Sheng\ Meng},
       DOI = {10.4310/jdg/1717356154},
       URL = {https://doi.org/10.4310/jdg/1717356154},
}

@article{Kollar1986-I,
    AUTHOR = {Koll\'ar, J\'anos},
     TITLE = {Higher direct images of dualizing sheaves. {I}},
   JOURNAL = {Ann. of Math. (2)},
  FJOURNAL = {Annals of Mathematics. Second Series},
    VOLUME = {123},
      YEAR = {1986},
    NUMBER = {1},
     PAGES = {11--42},
      ISSN = {0003-486X,1939-8980},
   MRCLASS = {14J10 (14F10 14J30 14J40)},
  MRNUMBER = {825838},
MRREVIEWER = {Takao\ Fujita},
       DOI = {10.2307/1971351},
       URL = {https://doi.org/10.2307/1971351},
}

@article{Kollar1986-II,
   AUTHOR = {Koll\'ar, J\'anos},
     TITLE = {Higher direct images of dualizing sheaves. {II}},
   JOURNAL = {Ann. of Math. (2)},
  FJOURNAL = {Annals of Mathematics. Second Series},
    VOLUME = {124},
      YEAR = {1986},
    NUMBER = {1},
     PAGES = {171--202},
      ISSN = {0003-486X,1939-8980},
   MRCLASS = {14F10 (14C30 14D99 32C35)},
  MRNUMBER = {847955},
MRREVIEWER = {Takao\ Fujita},
       DOI = {10.2307/1971390},
       URL = {https://doi.org/10.2307/1971390},
}

@incollection{Reid-1987-YPG,
  AUTHOR = {Reid, Miles},
     TITLE = {Young person's guide to canonical singularities},
 BOOKTITLE = {Algebraic geometry, {B}owdoin, 1985 ({B}runswick, {M}aine,
              1985)},
    SERIES = {Proc. Sympos. Pure Math.},
    VOLUME = {46, Part 1},
     PAGES = {345--414},
 PUBLISHER = {Amer. Math. Soc., Providence, RI},
      YEAR = {1987},
      ISBN = {0-8218-1476-1},
   MRCLASS = {14E30 (14B05 14E05 14J10)},
  MRNUMBER = {927963},
MRREVIEWER = {Eckart\ Viehweg},
       DOI = {10.1090/pspum/046.1/927963},
       URL = {https://doi.org/10.1090/pspum/046.1/927963},
}

@INPROCEEDINGS{Miyaoka1987,
   AUTHOR = {Miyaoka, Yoichi},
     TITLE = {The {C}hern classes and {K}odaira dimension of a minimal
              variety},
 BOOKTITLE = {Algebraic geometry, {S}endai, 1985},
    SERIES = {Adv. Stud. Pure Math.},
    VOLUME = {10},
     PAGES = {449--476},
 PUBLISHER = {North-Holland, Amsterdam},
      YEAR = {1987},
      ISBN = {0-444-70313-6},
   MRCLASS = {14E30 (14F05 14J40)},
  MRNUMBER = {946247},
MRREVIEWER = {I.\ Dolgachev},
       DOI = {10.2969/aspm/01010449},
       URL = {https://doi.org/10.2969/aspm/01010449},
}

@incollection{Catanese1999,
   AUTHOR = {Catanese, Fabrizio},
     TITLE = {Singular bidouble covers and the construction of interesting
              algebraic surfaces},
 BOOKTITLE = {Algebraic geometry: {H}irzebruch 70 ({W}arsaw, 1998)},
    SERIES = {Contemp. Math.},
    VOLUME = {241},
     PAGES = {97--120},
 PUBLISHER = {Amer. Math. Soc., Providence, RI},
      YEAR = {1999},
      ISBN = {0-8218-1149-5},
   MRCLASS = {14J29 (14J17 14J25)},
  MRNUMBER = {1718139},
MRREVIEWER = {Olivier\ Debarre},
       DOI = {10.1090/conm/241/03630},
       URL = {https://doi.org/10.1090/conm/241/03630},
}

@incollection {KMM87,
    AUTHOR = {Kawamata, Yujiro and Matsuda, Katsumi and Matsuki, Kenji},
     TITLE = {Introduction to the minimal model problem},
 BOOKTITLE = {Algebraic geometry, {S}endai, 1985},
    SERIES = {Adv. Stud. Pure Math.},
    VOLUME = {10},
     PAGES = {283--360},
 PUBLISHER = {North-Holland, Amsterdam},
      YEAR = {1987},
      ISBN = {0-444-70313-6},
   MRCLASS = {14E30 (14E05 14J40)},
  MRNUMBER = {946243},
MRREVIEWER = {David\ R.\ Morrison},
       DOI = {10.2969/aspm/01010283},
       URL = {https://doi.org/10.2969/aspm/01010283},
}

@article {Fuj77,
    AUTHOR = {Fujita, Takao},
     TITLE = {Some remarks on {K}odaira dimensions of fiber spaces},
   JOURNAL = {Proc. Japan Acad. Ser. A Math. Sci.},
  FJOURNAL = {Japan Academy. Proceedings. Series A. Mathematical Sciences},
    VOLUME = {53},
      YEAR = {1977},
    NUMBER = {1},
     PAGES = {28--30},
      ISSN = {0386-2194},
   MRCLASS = {32L05 (32J99)},
  MRNUMBER = {481131},
MRREVIEWER = {W.\ Barth},
       URL = {http://projecteuclid.org/euclid.pja/1195518154},
}

@article {Kaw92,
    AUTHOR = {Kawamata, Yujiro},
     TITLE = {Abundance theorem for minimal threefolds},
   JOURNAL = {Invent. Math.},
  FJOURNAL = {Inventiones Mathematicae},
    VOLUME = {108},
      YEAR = {1992},
    NUMBER = {2},
     PAGES = {229--246},
      ISSN = {0020-9910,1432-1297},
   MRCLASS = {14E30 (14E05 14E35)},
  MRNUMBER = {1161091},
MRREVIEWER = {Jaros\l aw\ A.\ Wi\'sniewski},
       DOI = {10.1007/BF02100604},
       URL = {https://doi.org/10.1007/BF02100604},
}

@article {T-Z1,
    AUTHOR = {Tsunoda, Shuichiro and Zhang, De-Qi},
     TITLE = {Noether's inequality for noncomplete algebraic surfaces of
              general type},
   JOURNAL = {Publ. Res. Inst. Math. Sci.},
  FJOURNAL = {Kyoto University. Research Institute for Mathematical
              Sciences. Publications},
    VOLUME = {28},
      YEAR = {1992},
    NUMBER = {1},
     PAGES = {21--38},
      ISSN = {0034-5318,1663-4926},
   MRCLASS = {14J29 (14C20 14J10)},
  MRNUMBER = {1147849},
MRREVIEWER = {Takao\ Fujita},
       DOI = {10.2977/prims/1195168854},
       URL = {https://doi.org/10.2977/prims/1195168854},
}

@article {Z2,
    AUTHOR = {Zhang, De-Qi},
     TITLE = {Noether's inequality for noncomplete algebraic surfaces of
              general type. {II}},
   JOURNAL = {Publ. Res. Inst. Math. Sci.},
  FJOURNAL = {Kyoto University. Research Institute for Mathematical
              Sciences. Publications},
    VOLUME = {28},
      YEAR = {1992},
    NUMBER = {5},
     PAGES = {679--707},
      ISSN = {0034-5318,1663-4926},
   MRCLASS = {14J10 (14C20 14J29)},
  MRNUMBER = {1195995},
MRREVIEWER = {Takao\ Fujita},
       DOI = {10.2977/prims/1195167932},
       URL = {https://doi.org/10.2977/prims/1195167932},
}

@book {Beauville,
    AUTHOR = {Beauville, Arnaud},
     TITLE = {Complex algebraic surfaces},
    SERIES = {London Mathematical Society Student Texts},
    VOLUME = {34},
   EDITION = {Second},
      NOTE = {Translated from the 1978 French original by R. Barlow, with
              assistance from N. I. Shepherd-Barron and M. Reid},
 PUBLISHER = {Cambridge University Press, Cambridge},
      YEAR = {1996},
     PAGES = {x+132},
      ISBN = {0-521-49510-5; 0-521-49842-2},
   MRCLASS = {14Jxx (14-02)},
  MRNUMBER = {1406314},
       DOI = {10.1017/CBO9780511623936},
       URL = {https://doi.org/10.1017/CBO9780511623936},
}

@article {EXPII,
    AUTHOR = {Chen, Jungkai A. and Chen, Meng},
     TITLE = {Explicit birational geometry of $3$-folds of general type, {II}},
   JOURNAL = {J. Differential Geom.},
  FJOURNAL = {Journal of Differential Geometry},
    VOLUME = {86},
      YEAR = {2010},
    NUMBER = {2},
     PAGES = {237--271},
      ISSN = {0022-040X,1945-743X},
   MRCLASS = {14E05 (14J30)},
  MRNUMBER = {2772551},
MRREVIEWER = {Massimiliano\ Mella},
       URL = {http://projecteuclid.org/euclid.jdg/1299766788},
}

@Article{HZ,
  author   = {Hu, Yong and Zhang, Tong},
  journal  = {Math. Ann.},
  title    = {Algebraic threefolds of general type with small volume},
  year     = {2025},
  issn     = {0025-5831},
  number   = {1},
  pages    = {567--612},
  volume   = {391},
  doi      = {10.1007/s00208-024-02933-6},
  fjournal = {Mathematische Annalen},
  mrclass  = {14J30 (14E30 14J29)},
  mrnumber = {4846792},
  url      = {https://doi.org/10.1007/s00208-024-02933-6},
}

@Article{CHPZ,
  author        = {Coughlan, Stephen and Hu, Yong and Pignatelli, Roberto and Zhang, Tong },
  journal       = {arXiv: 2409.17847v2},
  title         = {Moduli spaces of threefolds on the {N}oether line},
  year          = {2024},
  archiveprefix = {arXiv},
  eprint        = {2409.17847v2},
  primaryclass  = {math.AG},
}

\end{document}